%% file: ms.tex
\documentclass[paper=a4,fontsize=11pt,twoside]{scrartcl}  
\title{Logical coherence in Bayesian simultaneous three-way hypothesis tests}

\author{Bernardo F. Reimann \and Rafael Izbicki \and Julio M. Stern \and Rafael B. Stern \and  Lu\'is G. Esteves}

\input{config}

\begin{document}
\maketitle

\begin{abstract}
{\textbf{Abstract:}}
This paper studies whether 
Bayesian simultaneous three-way hypothesis tests
can be logically coherent.
Two types of results are obtained.
First, under the standard error-wise constant loss,
only for a limited set of models can
a Bayes simultaneous test be logically coherent.
Second, if more general loss functions are used, then
it is possible to obtain Bayes simultaneous tests
that are always logically coherent.
An explicit example of such a loss function is provided.
\end{abstract}

\input{sections/01-intro}
\input{sections/02-review}
\input{sections/03-ec-loss}
\input{sections/04-general}
\input{sections/05-conclusion}

\bibliography{main}

\include{sections/A1-proofs}

\end{document}

%% file: config.tex
\usepackage[T1]{fontenc}
\usepackage{nameref}
\usepackage[colorlinks=true, urlcolor=blue, citecolor=blue, linkcolor=blue]{hyperref}
\usepackage{amsfonts, amsmath, amssymb, amsthm, thmtools}
\usepackage{mathrsfs}
\usepackage{cleveref}

\declaretheorem[name=Theorem, refname={Theorem, Theorems}, Refname={Theorem, Theorems}, parent=section]{theorem}

\declaretheorem[name=Definition, refname={Definition, Definitions}, Refname={Definition, Definitions}, sibling=theorem, style=definition]{definition}
\declaretheorem[name=Example, refname={Example, Examples}, Refname={Example, Examples}, sibling=theorem, style=definition]{example}

\declaretheorem[name=Lemma, refname={Lemma, Lemmas}, Refname={Lemma, Lemmas}, sibling=theorem]{lemma}

\crefname{section}{section}{sections}
\Crefname{section}{Section}{Sections}
\crefname{table}{table}{tables}
\Crefname{table}{Table}{Tables}

\usepackage{color, enumitem, fancyhdr, float, fourier, layout, multirow, setspace, verbatim}
\setlist[enumerate]{leftmargin=*}
\usepackage[colorinlistoftodos]{todonotes}
\usepackage{algorithm2e, algorithmicx, listings}
\usepackage{caption, subcaption}
\usepackage{natbib}
\usepackage{enumitem}
\usepackage{graphicx}
\graphicspath{{figures/}{../figures/}}

\def\half{\frac{1}{2}}

\def\I{{\mathbb I}}

\def\E{{\textbf{E}}}

\def\P{{\mathbb P}}

\def\sP{\mathcal{P}}
\def\sX{\mathcal{X}}

\def\t0{{\theta_0}}

\usepackage{setspace}
\definecolor{darkgreen}{rgb}{0.0, 0.5, 0.0}

\usepackage[normalem]{ulem}
\newif\ifdraft
\drafttrue

\newcommand\remove{\bgroup\markoverwith{\textcolor{gray}{\rule[.5ex]{2pt}{1pt}}}\ULon}

\newcommand{\gfbstfig}{
 \begin{tikzpicture}[thick,scale=0.75, every node/.style={scale=0.9}]
  \draw (-6,2) circle (1.5);
  \draw [fill=lightgray] (-6,2) circle (0.8);
  \draw (-8,0) -- (-4,0) -- (-4,4) -- (-8,4) -- (-8,0);;
  \node at (-6,2) {{\large $R(x)$}}; 
  \node at (-6,0.8) {{\large $H$}};
  \node at (-6,4.5) {{\large $x \in POS(H)$}};
			
  \draw (-1.4,2.4) circle (1.1);
  \draw [fill=lightgray] (0.2,0.9) circle (0.7);
  \draw (-3,0) -- (1,0) -- (1,4) -- (-3,4) -- (-3,0);;
  \node at (0.2,0.9) {{\large$R(x)$}}; 
  \node at (-1.4,2.4) {{\large $H$}};
  \node at (-1,4.5) {{\large $x \in NEG(H)$}};	
			
  \draw [fill=lightgray] (5.2,1.9) circle (0.7); 
  \draw (4.4,2.4) circle (1.1);
  \draw (2,0) -- (6,0) -- (6,4) -- (2,4) -- (2,0);;
  \node at (5.2,1.9) {{\large $R(x)$}}; 
  \node at (4.2,2.5) {{\large $H$}};
  \node at (4,4.5) {{\large $x \in BND(H)$}};	
 \end{tikzpicture}
}

%% file: sections/01-intro.tex
\section{Introduction}

In a three-way decision problem 
\citep{Yao2012,Liu2014,Yao2015}
one must classify objects into 
three categories.
While a two-way decision necessarily 
leads to an affirmation or a negation,
a three-way decision also allows
non-commitment or
pause to gather more evidence.
Such a flexible approach has led
to advances in areas such as
clustering \citep{Yu2017},
classification \citep{Zhou2014,Zhang2019},
multi-agent decisions \citep{Yang2012}
game theory \citep{Herbert2011,Azam2014,Bashir2021}, and
recommender systems \citep{Zhang2017}.

In particular, three-way decisions can
also be applied to statistical hypothesis testing
\citep{Wald1945,Kaiser1960,Tukey1960,Harris2016,Berg2004,Goudey2007,Esteves2016,Stern2017}.
In this context, one gathers data, $x \in \sX$, to
decide whether an unobserved quantity, $\theta \in \Theta$,
satisfies $\theta \in H$, for $H \subseteq \Theta$. While standard hypothesis tests allow only the rejection or non-rejection of $H$, three-way (agnostic) tests allow $H$ to be accepted, rejected or remain undecided. In the statistical literature, such a decision is usually represented by a function, $\varphi_{H}: \sX \rightarrow \{0,\half,1\}$. In this context, $\varphi_{H}(x) = 0$, $\varphi_{H}(x) = 1$, and $\varphi_{H}(x) = \half$ mean that one decides to, respectively, 
accept, reject and remain undecided about $H$
after observing $x$. This definition can be identified with
the standard three-decision regions:
\begin{align*}
 POS(H_0)
 &:= \{x \in \sX: \varphi_{H}(x) = 0\} \\
 NEG(H_0) 
 &:= \{x \in \sX: \varphi_{H}(x) = 1\} \\
 BND(H_0) 
 &:= \left\{x \in \sX: \varphi_{H}(x) = \half\right\}
\end{align*}

In order to determine the optimal decision regions, 
one can use Bayesian decision theory \citep{Yao2010}.
In this context, one possible approach is 
to use an error-constant (EC) loss function (Definition \ref{def:ec_loss}),
as presented in \cref{ex:ec_loss}.

\begin{definition}[Error-wise constant loss function]
 \label{def:ec_loss}
 Let $H$ be an hypothesis. 
 The error-wise constant (EC) loss function,
 $L_H$, is given by \cref{tab:ec_loss},
 where $0 < \lambda_{BP}^H < \frac{\lambda_{NP}^H}{2}$,
 $0 < \lambda_{BN}^H < \frac{\lambda_{PN}^H}{2}$, and
 $(\lambda_{PN}^H-\lambda_{BN}^H) \lambda_{NP}^H 
 > \lambda_{BP}^H\lambda_{PN}^H$.
 These restrictions are made so that
 the loss for each type of error corresponds to
 its intuitive meaning. For instance, when $H$,
 accepting $H$ is better than
 not deciding, which in turn is 
 better than rejecting $H$.
 Also, not deciding is always better than
 deciding randomly between 
 accepting or rejecting $H$.
 \begin{table}
  \centering
  \begin{tabular}{l|ll}
   & $\theta \in H$ & $\theta \notin H$ \\
   \hline
   accept    & $0$ & $\lambda_{PN}^H$ \\
   undecided & $\lambda_{BP}^H$ & $\lambda_{BN}^H$ \\
   reject    & $\lambda_{NP}^H$ & $0$
  \end{tabular}
  \caption{Error-constant loss function}
  \label{tab:ec_loss}
 \end{table}
\end{definition}

\begin{example}[Posterior probability three-way tests]
 \label{ex:ec_loss}
 Under the EC loss (\cref{def:ec_loss}),
 \citet{Yao2007} determines 
 the optimal three-way decision regions
 for hypothesis tests:
 \begin{align}
  \label{eq:ec_bayes}
  POS(H) 
  &= \{x \in \sX: \P(\theta \in H|x) > \beta^H\}, 
  \nonumber \\
  NEG(H)
  &= \{x \in \sX: \P(\theta \in H|x) < \alpha^H\}
  , \text{ and } 
  \nonumber \\
  BND(H)
  &= \{x \in \sX: \alpha^H \leq \P(\theta \in H|x) \leq \beta^H\}, 
  \nonumber \\
  \text{where } 
  & \beta^H = \frac{\lambda_{PN}^H-\lambda_{BN}^H}
  {(\lambda_{PN}^H-\lambda_{BN}^H)+\lambda_{BP}^H} < 1,
  \text{ and } 
  \alpha^H = \frac{\lambda_{BN}^H} 
  {(\lambda_{BN}^H-\lambda_{BP}^H)+\lambda_{NP}^H} > 0.
 \end{align}
\end{example}

A more general setting occurs in 
simultaneous hypothesis testing, in which 
one wishes to test a collection of hypotheses, 
$\sigma(\Theta)$, at the same time \citep{shaffer1995multiple,lehmann2005testing}.
\cref{def:bayes} describes Bayesian optimality
in this context:

\begin{definition}[Bayesian optimality for simultaneous hypothesis tests]
 \label{def:bayes}
 For each hypothesis, $H \in \sigma(\Theta)$, 
 let $L_H$ be a loss function. 
 A simultaneous hypothesis test, $\varphi$, is
 Bayes with respect to $L$ if,
 for every hypothesis, $H$,
 $\varphi_H$ is a Bayes test for testing
 $H$ against $L_H$.
\end{definition}

The following example shows that
posterior-probability based simultaneous tests are
obtained from the EC loss 
in a similar fashion as in \cref{ex:ec_loss}:

\begin{example}[Simultaneous test based 
for error-wise constant (EC) losses]
 \label{ex:post_prob}
 Let $L$ be a loss function such that,
 for each hypothesis, $H$, $L_H$ is
 the loss function presented in \cref{tab:ec_loss}.
 In this case, the simultaneous test that
 satisfies \cref{eq:ec_bayes} for each $H$ is
 Bayes with respect to $L$.
\end{example}

\begin{definition}[Simultaneous test based for trivial error-wise constant (TEC) losses]
 \label{def:post_prob}
 If for each $H \in \sigma(\Theta)$,
 $L_H$ is such that the constants in
 \cref{tab:ec_loss} do not depend on $H$, then
 $L$ is said to be a trivial error-wise constant loss (TEC).
 In this case, the Bayes simultaneous test given by
 \cref{eq:ec_bayes} is such that 
 $\alpha^H$ and $\beta^H$ do not depend on $H$.
\end{definition}

In the context of simultaneous tests,
one is often interested in 
an overall interpretation of
all the tests.
One condition that is required for
the interpretability of the tests is
their logical coherence.
For instance, if $x \in POS(\theta > 1)$ and
also $x \in NEG(\theta > 0)$, then,
after observing $x$, one would believe both
that $``\theta > 1''$ is true and that 
$``\theta > 0''$ is false,
a logical contradiction.
Such contradictory conclusions
are hard to interpret and
should be avoided.

Based on this challenge and
on previous proposals for logical requirements \citep{Gabriel1969,Schervish1996,Lavine1999,Hommel2008,Romano2011,Izbicki2015,Hansen2022}, 
the concept of logical coherence in
simultaneous hypothesis testing
is proposed \citep{Esteves2016}:

\begin{definition}[Logical coherence]
 \label{def:coherence}
 A simultaneous hypothesis test is 
 logically coherent if:
 \begin{enumerate}
  \item {(Propriety)}\ $POS(\Theta) = \sX$,
  \item {(Monotonicity)}\ If $H_1 \subseteq H_2$, then
  $x \in POS(H_1)$ implies that $x \in POS(H_2)$ and
  $x \in BND(H_1)$ implies that
  $x \in BND(H_2) \cup POS(H_2)$ ,
  \item {(Intersection consonance)} \ If $x \in POS(H_1)$ and $x \in POS(H_2)$, then
  $x \in POS(H_1 \cap H_2)$ ,
  \item {(Invertibility)}\  If $x \in POS(H)$, then
  $x \in NEG(H^c)$ .
 \end{enumerate}
\end{definition}

This paper studies under what conditions
it is possible to obtain 
a Bayes simultaneous test that
is logically coherent.
\Cref{sec:review} reviews a useful
characterization of logical coherence
in terms of region estimators.
Using this characterization, 
\Cref{sec:ec} explores the relation between
the EC loss and logical coherence.
This section shows that it is impossible to
fully reconcile Bayesian decision theory with
logical coherence while using the EC loss.
Given this impossibility, 
\Cref{sec:gfbst} explores more general loss functions.
This section defines the GFBST loss and shows that,
under this loss, the Bayes test is
always logically coherent.

%% file: sections/02-review.tex
\section{Characterization of logical coherence}
\label{sec:review}

Logically coherent tests can be characterized
in terms of region estimators \citep{Esteves2016},
A region estimator, $R(x)$, is usually interpreted as
a set of likely values for $\theta$.
Region estimators are formalized below:

\begin{definition}[Region estimator]
 \label{def:region}
 A region estimator is a function
 $R: \sX \rightarrow \sP(\Theta)$,
 where $\sP(\Theta)$ is the collection of
 all subsets of $\Theta$.
\end{definition}

A particular type of region estimator is the
highest posterior density (HPD) set.
The HPD contains the parameter values with
posterior density above a given threshold.
If $\Theta$ is finite, then
the posterior density is often
taken as the posterior probability, that is,
the HPD contains the most probable values for $\theta$.

\begin{example}[Highest posterior density set]
 \label{example:hpd}
 A region estimator, $R(x)$, is a 
 highest posterior density set with
 respect to a posterior density, $f(\theta|x)$, if
 there exists $k$ such that
 \begin{align*}
  R(x) &=
  \{\theta \in \Theta: 
  f(\theta|x) \geq k\}.
 \end{align*}
\end{example}

\begin{figure}
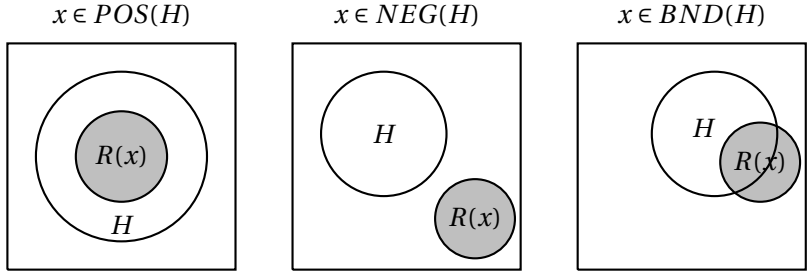

  \center
  \gfbstfig
  \mbox{} \vspace{-2mm} \mbox{} \\ 
  \caption{$\varphi$ is
  a region-based test 
  for testing $H$.}
  \label{fig::region}
\end{figure}

Using region estimators,
one can construct a simultaneous test,
as illustrated in \cref{fig::region}.
A test based on a region estimator, $R(x)$, 
accepts $H$ if 
$R(x) \subseteq H$, that is,
all likely values for $\theta$ reside in $H$.
Similarly, it reject $H$ if
$H \cap R(x) = \emptyset$, that is
no likely value of $\theta$ resides in $H$.
Otherwise, the test remains agnostic about $H$.

\begin{definition}[Region-based test] 
 \label{def:region_test}
 $\varphi$ is a region-based test if
 there exists a region estimator, $R$,
 such that $x \in POS(H)$ if $R(x) \subseteq H$,
 $x \in NEG(H)$ if $R(x) \cap H = \emptyset$, and
 $x \in BND(H)$, otherwise, that is,
 \begin{align*}
  \varphi_H(x) &=
  \begin{cases}
   0 & \text{, if } R(x) \subseteq H \\
   1 & \text{, if } R(x) \cap H = \emptyset \\
   \half & \text{, otherwise}
  \end{cases}
 \end{align*}
\end{definition}

The (non-invariant)
Generalized Full Bayesian Significance Test (GFBST; \citealt{Stern2017}) is
a particular type of test based on a region estimator.
It uses an HPD as region estimator.

\begin{example}[GFBST]
 \label{ex:GFBST}
 The GFBST is a region-based test
 in which $R(x)$ is an HPD.
\end{example}

\Cref{ex:hypergeometric} describes a GFBST.

\begin{example}
 \label{ex:hypergeometric}
 Consider that $n$ balls are removed 
 without replacement from a box with
 with $\theta$ blue balls and
 $N-\theta$ yellow balls,
 where $n < N$.
 The total number of sampled blue balls, $X$,
 follows $X|\theta \sim \text{Hypergeometric}(\theta,N-\theta,n)$.
 Also, consider that, a priori,
 $\theta \sim \text{Binomial}(N, 0.5)$.
 It can be shown that
 $\theta - X|X \sim Binomial(N - n, 0.5)$. 
 Hence, for every $k > 0$ and $1 \leq x \leq n$,
 \begin{align*}
  R_k(x) = \left\{0 \leq i \leq N: 
  \bigg|\frac{(N-n)}{2}+x-i\bigg| < k\right\}
  & \text{ is a HPD.}
 \end{align*}
 In this case, a GFBST accepts $H$ if 
 it contains all points close to $\frac{N-n}{2} + x$,
 rejects $H$ if it contains none of these points,
 and  otherwise remains agnostic.
\end{example}

Under special circumstances
all logically coherent simultaneous tests are
based on region estimators \citep{Esteves2016}.
In particular, this relation is valid
when $\Theta$ is a finite set:

\begin{theorem}
 \label{thm:logical_region}
 If $\Theta$ is finite and
 $\varphi$ is a logically coherent
 simultaneous test, then
 $\varphi$ is based on a region estimator.
\end{theorem}

The next section studies under
what circumstances a Bayes test against 
an EC loss can be logically coherent.

%% file: sections/03-ec-loss.tex
\section{The relation between Bayesian optimality and logical coherence under error-wise constant loss}
\label{sec:ec}

A logically coherent test, $\varphi$, that is
Bayes against an EC loss admits
further characterization.
In such a case, not only is $\varphi$ 
a region-based test, but also
based on an HPD. That is,
a logically coherent test that is 
Bayes against an EC loss is a GFBST,
as presented in \cref{thm:GFBSTTEC}.\footnote{\cref{thm:consistentRegion}, in the Appendix, is used to prove \cref{thm:GFBSTTEC}.
Recall that if a test is logically coherent and $\Theta$ is a finite set, then the test is based on a region estimator.
\cref{thm:consistentRegion} shows that, if a Bayes test is
based on a region estimator, then
there exists a loss such that the region estimator is Bayes.
That is, a Bayes logically coherent test is necessarily based on a region estimator which is also Bayes.}

\begin{theorem}
\label{thm:GFBSTTEC}
 Let $\Theta$ be a finite set.
 If there exists a probability, $\P$, and 
 a TEC loss, $L$, such that
 a logically coherent simultaneous test, $\varphi$, is 
 Bayes against $L$
 according to $\P$, then $\varphi$ is a GFBST.
\end{theorem}

However, do there exist actual cases
in which a test is both 
Bayes with respect to a TEC loss and 
also logically coherent?
\Cref{sec:review} shows that
a logically coherent test must be
based on a region estimator.
Also, \citet{Yao2007} shows that
a Bayes test against a TEC loss must be
a probability-based test.
Despite these strong restrictions,
\cref{thm:fullyconsistent_EC} shows that
every logically coherent test is
Bayes against a TEC loss for 
some probability measure.

\begin{theorem}
 \label{thm:fullyconsistent_EC}
 Let $\Theta$ and $\sX$ be finite sets.
 If $\phi$ is a logically coherent simultaneous test, then
 there exists a probability, $\P$, and
 a TEC loss function, $L$, such that
 $\varphi$ is Bayes against $L$.
\end{theorem}

\Cref{thm:fullyconsistent_EC} shows that, 
for each logically coherent test,
there exists a choice of $\P$ and $L$ such that
the test is also Bayes with respect to a TEC loss.\footnote{Under mild assumptions, \cref{thm:GFBSTTEC,thm:fullyconsistent_EC} also hold when $\Theta$ is a countable set.}
\Cref{ex:tec_region} shows a choice of $L$ and $\P$ so that
a logically coherent test is Bayes and, therefore,
also is a GFBST.

\begin{example}
 \label{ex:tec_region}
 Let $\Theta = \{1,2,3,4\}$,
 $\sX \in \{0,1\}$,
 $R(0) = \{1,2\}$,
 $R(1) = \{3,4\}$, and
 $\varphi$ be a test based on $R$.
 Let $L$ be a TEC loss so that
 $\beta = \frac{7}{10}$ and
 $\alpha = \frac{3}{10}$.
 Also, let 
 $\P(1|x) = \P(2|x) = \frac{4^{1-x}}{10}$ and
 $\P(3|x) = \P(4|x) = \frac{4^{x}}{10}$.
 Let $\varphi^*$ be the Bayes test according to $L$.
 Note that the two least probable outcomes
 sum up a probability of $\frac{2}{10}$.
 Hence, every hypothesis that contains 
 none of the most probable outcomes is rejected by $\varphi^*$.
 Next, if an hypothesis contains both 
 of the most probable outcomes, than 
 its probability is at least $\frac{8}{10}$, 
 so it is accepted by $\varphi^*$.
 Finally, if an hypothesis contains only one
 of the most probable outcomes, than 
 its probability is between 
 $\frac{4}{10}$ and $\frac{6}{10}$,
 so $\varphi^*$ remains agnostic about $H$.
 From the previous conclusions, obtain that
 $\varphi \equiv \varphi^*$, that is,
 $\varphi$ is a logically coherent test that is
 Bayes against $L$ according to $\P$.
 Finally, note that when using $\P$,
 $R$ is an HPD, that is, $\varphi$ is a GFBST,
 as also known from \cref{thm:GFBSTTEC}.
\end{example}

\Cref{ex:tec_region} shows that, for a given 
region-based test, a specific choice of TEC loss and $\P$
are required so that the test is Bayes.
However, in most settings $\P$ is given and
one wishes to choose $L$ so that
the Bayes test is logically coherent.
\Cref{thm:impossibility} shows that there is
no choice of an EC loss such that 
the Bayes test is logically coherent for every $\P$.

\begin{theorem}
 \label{thm:impossibility}
 Let $|\Theta| \geq 3$.
 For each $\P$ and $L$, let
 $\varphi_{\P,L}$ be a 
 Bayes simultaneous test against $L$
 according to $\P$.
 If $L$ is an EC loss, then
 there exists \ $\P$ such that 
 $\varphi_{\P,L}$ is not logically coherent.
\end{theorem}

\Cref{thm:impossibility} shows that,
if $L$ is an EC loss, then
there exists a probability, $\P$, such that
the resulting Bayes test 
is not logically coherent.
Hence, a procedure that yields
Bayes tests that are logically coherent 
must be based on more general loss functions.
The next section explores these losses.

%% file: sections/04-general.tex
\section{A logically coherent Bayesian procedure}  
\label{sec:gfbst}

This section develops a loss function such that,
for every probability, $\P$,
the resulting Bayes test is logically coherent.
This loss is presented in \cref{def:gfbst_loss}:

\begin{definition}[GFBST loss]
 \label{def:gfbst_loss}
 Let $\mu$ be a measure over $\Theta$ such that
 $\P(\theta|x)$ is absolutely continuous
 with respect to $\mu$ for every $x \in \sX$ and
 $f(\theta|x) := \frac{dP(\theta|x)}{d\mu}$.
 The tangent set to hypothesis $H$
 according to $\mu$, $T_x^H$, 
 is defined as
 $T_x^H := \{\theta \in \Theta : 
 f(\theta | x)  >  \sup_{\theta' \in H} f(\theta' | x) \}$.
 The GFBST loss according to $\mu$ for
 testing $H$ is given by \Cref{tab:gfbstloss}.
 \begin{table}
  \centering
  \begin{tabular}{ l|c c c}
   \textbf{decision} & & \textbf{state of the nature} \\
   & $\theta \in T_x^H $ 
   & $\theta \notin T_x^H \cup T_x^{H^c}$ 
   & $\theta \in T_x^{H^c}$ \\ \hline
   $0$  & $b + c$ & $b$ & $0$  \\
   ${\frac{1}{2}}$    &  $v + c$ & $v$ & $v + c$  \\ 
   $1$  & $0$ & $b$   & $b + c$  \\
  \end{tabular}
  \caption{The GFBST loss.}
  \label{tab:gfbstloss}
 \end{table}
\end{definition}

The GFBST loss, which 
generalizes the two-way counterpart in \citet{Madruga2001}, 
admits an intuitive interpretation \citep{Stern2003}.
Observe that $T_{x}^{H} \subseteq H^c$ is 
the collection of values in $\Theta$ that 
are more likely than every point in $H$. 
Hence, $T_x^H$ and $T_x^{H^c}$ can be interpreted as 
the set of points that are strong contenders for,
respectively, $H$ and $H^c$.
The GFBST loss is lowest, $0$, when either 
$H$ is rejected and $\theta$ is a strong contender for $H$ or
$H$ is accepted and $\theta$ is a strong contender for $H^c$.
Also the GFBST is largest, $b+c$, when either 
$H$ is rejected and $\theta$ is a strong contender for $H^c$ or
$H$ is accepted and $\theta$ is a strong contender for $H$.
Finally, the GFBST loss assumes intermediate values, when either
$\theta$ is not a strong contender for $H$ or $H^c$ or when
the agnostic decision is chosen.

\begin{theorem}
 \label{thm:gfbstloss}
 For every probability, $\P$,
 if $\varphi$ is a Bayes simultaneous test
 against the GFBST loss, then
 $\varphi$ is a GFBST.
\end{theorem}

\Cref{thm:gfbstloss} shows that,
if the GFBST loss is used, then
the Bayes test is a GFBST.
Therefore, for every probability measure,
the Bayes test against the GFBST loss is
logically coherent.
Hence, using loss functions that are more general than
the EC loss, it is possible to always reconcile
Bayesian decision theory with logical coherence.

%% file: sections/05-conclusion.tex
\section{Final remarks} 
\label{sec:final}

Simultaneous three-way decisions may 
require more constraints than
are typically used in 
individual decisions.
In particular, 
when performing simultaneous hypothesis test,
one might expect logical coherence 
between conclusions.
This paper presents results on whether
it possible to obtain logical coherence
together with Bayesian optimality.

Two types of results are obtained.
If an error-wise constant loss is used, then
only for a limited set of models can
a Bayes simultaneous test be logically coherent.
This result motivated the investigation
of other types of loss functions
which might provide a better reconciliation between 
Bayesian optimality and logical coherence.
We propose the GFBST loss and show that
every Bayes test against this loss is a GFBST.
Since every GFBST is logically coherent,
the GFBST loss yields Bayes tests that
are always logically coherent.

The above results show that
the GFBST loss can lead to simultaneous tests that
yield conclusions which are more interpretable than 
the ones obtained from the EC loss.
The results also show that 
simultaneous three-way decisions can
yield a layer of complexity that 
is not present in individual decision problems.
Further investigation might determine whether
this layer of complexity is also present
in other applications of three-way decisions, 
such as classification or clustering.

%% file: sections/A1-proofs.tex
\appendix

\section{Proofs}

\begin{proof}[Proof of \cref{thm:logical_region}]
 Let $\mathcal{F} = \{H \in \sigma(\Theta): 
 \forall H^* \in \sigma(\Theta), 
 H \cap H^* \in \{\emptyset, H\}\}$.
 Since $\Theta$ is finite and
 $\sigma(\Theta)$ is a $\sigma$-field,
 $\mathcal{F}$ partitions $\Theta$.
 Define the equivalence relation 
 $\sim$ such that $\theta_1 \sim \theta_2$ if
 there exists $F \in \mathcal{F}$ such that
 $\theta_1 \in F$ and $\theta_2 \in F$.
 Define $\Theta^*$ as the quotient space
 $\Theta\backslash\sim$. Also,
 let $\sigma(\Theta^*)$ and $\varphi^*$ be 
 the quotient $\sigma$-field of $\sigma(\Theta)$
 and the quotient test of $\varphi$ over $\sim$.
 It follows from construction that
 $\sigma(\Theta^*)$ includes the singleton.
 Hence, \citet{Esteves2016} obtains that
 $\varphi^*$ is based on a region estimator, $R^*$.
 Conclude that $\varphi$ is based on
 a region estimator, $R$.
\end{proof}

\begin{definition}[\textbf{Proper loss function}]
 \label{def:proper}
 A loss functions, $L$, is 
 proper if,
 for every $A \in \sigma(\Theta)$,
 \begin{align*}
  L_A(0,\theta) < L_A\left(\half,\theta\right) < L_A(1,\theta)  
  & \text{, if } \theta \in A \\
  L_A(0,\theta) > L_A\left(\half,\theta\right) > L_A(1,\theta)  
  &  \text{, if } \theta \notin A \\
  L_A\left(\half,\theta\right) < \frac{L_A(0,\theta) + L_A(1,\theta)}{2}
  &  \text{, } \forall \text{  } \theta \in \Theta
 \end{align*}
\end{definition}

\begin{lemma}
 \label{lemma:proper}
 If $L$ is a proper loss, then
 $\min\left(\E\left[L_{\{\theta^{'}\}}\left(\half,\theta\right)
  \bigg|x\right], (\E\left[L_{\{\theta^{'}\}}\left(0,\theta\right)
  \bigg|x\right]\right)
  \leq \E\left[L_{\{\theta^{'}\}}\left(1,\theta\right) \bigg|x\right]$ implies that
  $\E\left[L_{\{\theta^{'}\}}\left(\half,\theta\right)
  \bigg|x\right]
  \leq \E\left[L_{\{\theta^{'}\}}\left(1,\theta\right) \bigg|x\right]$.
\end{lemma}

\begin{proof}
 It is sufficient to prove that,
 if $\E\left[L_{\{\theta^{'}\}}\left(0,\theta\right)
  \bigg|x\right]
  \leq \E\left[L_{\{\theta^{'}\}}\left(1,\theta\right) \bigg|x\right]$, then
  $\E\left[L_{\{\theta^{'}\}}\left(\half,\theta\right)
  \bigg|x\right]
  \leq \E\left[L_{\{\theta^{'}\}}\left(1,\theta\right) \bigg|x\right]$. 
  Let $\E\left[L_{\{\theta^{'}\}}\left(0,\theta\right)
  \bigg|x\right]
  \leq \E\left[L_{\{\theta^{'}\}}\left(1,\theta\right) \bigg|x\right]$.
  Since $L$ is proper,
 \begin{align*}
  \E\left[L_{\{\theta^{'}\}}\left(\half,\theta\right)
  \bigg|x\right] 
  &\leq \frac{\E\left[L_{\{\theta^{'}\}}\left(0,\theta\right)\bigg|x\right]}{2} 
  + \frac{\E\left[L_{\{\theta^{'}\}}\left(1,\theta\right)
  \bigg|x\right]}{2} \\
  &\leq \E\left[L_{\{\theta^{'}\}}\left(1,\theta\right)
  \bigg|x\right].
 \end{align*}
\end{proof}

\begin{lemma}
\label{thm:consistentRegion}
 Let $\Theta$ be finite,
 $\sigma(\Theta)$ include the unitary sets, and
 $\varphi$ be generated by
 the region estimator, $R$.
 If there exists a probability, $\P$, and
 a proper loss, $L$, such that
 $\varphi$ is Bayes against $L$ according to $\P$, then
 $R$ is a Bayes region estimator 
 against $\bar{L}$ according to $\P$, where
 \begin{align*}
  \bar{L}(A,\theta)
  &= \sum_{\theta^{'} \in A}
  {\left[
  L_{\{\theta^{'}\}}\left(\half,\theta\right)
  -L_{\{\theta^{'}\}}\left(1,\theta\right)
  \right].}
 \end{align*}
\end{lemma}

\begin{proof}
 The Bayes region estimator against $\bar{L}$, 
 $R^*$, satisfies:
 \begin{align*}
  R^*(x) := \left\{\theta^{'} \in \Theta:
  \E\left[L_{\{\theta^{'}\}}\left(\half,\theta\right)
  \bigg|x\right]
  \leq \E\left[L_{\{\theta^{'}\}}\left(1,\theta\right)
  \bigg|x\right]
  \right\}.
 \end{align*}
 Hence, it is sufficient to prove that
 $R \equiv R^{*}$.
 Since $\varphi$ is Bayes against $L$,
 $\varphi_{\{\theta^{'}\}}(x) < 1$ if and only if
 $\min\left(\E\left[L_{\{\theta^{'}\}}\left(\half,\theta\right)
  \bigg|x\right], (\E\left[L_{\{\theta^{'}\}}\left(0,\theta\right)
  \bigg|x\right]\right)
  < \E\left[L_{\{\theta^{'}\}}\left(1,\theta\right) \bigg|x\right]$. Using \cref{lemma:proper},
  conclude that 
  $\varphi_{\{\theta^{'}\}}(x) < 1$ if and only if
  $\E\left[L_{\{\theta^{'}\}}\left(\half,\theta\right)
  \bigg|x\right]
  \leq \E\left[L_{\{\theta^{'}\}}\left(1,\theta\right) \bigg|x\right]$.
  Since $\varphi$ is generated by $R$, 
 it follows that 
 $R(x) = \{\theta^{'}: \varphi_{\{\theta^{'}\}}(x) < 1\}$,
 that is,
 \begin{align*}
  R(x) &= \left\{\theta^{'}: \varphi_{\{\theta^{'}\}}(x) < 1\right\} \\
  &= \left\{\theta^{'}: \E\left[L_{\{\theta^{'}\}}\left(\half,\theta\right)
  \bigg|x\right]
  \leq \E\left[L_{\{\theta^{'}\}}\left(1,\theta\right) \bigg|x\right]\right\} \equiv R^*(x)
 \end{align*}
\end{proof}

\begin{proof}[Proof of \cref{thm:GFBSTTEC}]
 Since $\varphi$ is logically coherent,
 it follows from \cref{thm:logical_region} that
 $\varphi$ is based on a region estimator, $R$.
 It follows from \cref{thm:consistentRegion} that
 $R$ is a Bayes region estimator against 
 $\overline{L}$. Since $L$ is a TEC loss,
 which is proper,
 $\overline{L}(A, \theta) = \lambda_{BN}|A| 
 -((\lambda_{BN} - \lambda_{BP})
 + \lambda_{NP})\I_A(\theta)$. That is,
 \begin{align*}
  R(x) &= \left\{\theta \in \Theta:
  \P(\theta|x) \geq \frac{\lambda_{BN}}{(\lambda_{BN} - \lambda_{BP}) + \lambda_{NP}}\right\}.
 \end{align*}
 Conclude that $R(x)$ is a HPD.
\end{proof}

\begin{lemma}[Union consonance]
 \label{lemma:union}
 Let $\varphi$ be logically coherent.
 If $H_1$ and $H_2$ are such that
 $\varphi_{H_1}(x) = 1$ and
 $\varphi_{H_2}(x) = 1$, then
 $\varphi_{H_1 \cup H_2}(x) = 1$.
\end{lemma}

\begin{proof}
 It follows from invertibility that
 $\varphi_{H_1^c}(x) = 0$ and
 $\varphi_{H_2^c}(x) = 0$.
 Hence, from intersection consonance,
 $\varphi_{H_1^c \cap H_2^c}(x) = 0$.
 Finally, conclude from invertibility that
 $\varphi_{H_1 \cup H_2}(x) = 1$.
\end{proof}

\begin{lemma}
\label{lemma:exisstenceagnostic}
 Let $\Theta$ be a finite set.
 If $\varphi$ is a logically coherent 
 simultaneous test, then:
 \begin{enumerate}[label=(\alph*)]
  \item For every $x \in \sX$, 
  there exists $\theta_0 \in \Theta$ such that
  $\varphi_{\{\theta_0\}}(x) < 1$.
  \item For every $x \in \sX$, 
  if $\varphi_{\{\theta_0\}}(x) = 0$, then
  $\varphi_{\{\theta\}}(x) = 1$, 
  $\forall \theta \neq \theta_0$.
 \end{enumerate}
\end{lemma}

\begin{proof}
\textbf{(a)} Assume that there exists
$x \in \sX$ such that $\varphi_{\{\theta\}}(x) = 1$,
for every $\theta \in \Theta$. It follows from
\cref{lemma:union} that $\varphi_{\Theta}(x) = 1$,
which contradicts the propriety of $\varphi$.
\textbf{(b)} Let $\theta_0$ be such that
$\varphi_{\{\theta_0\}}(x) = 0$. It follows from
invertibility that $\varphi_{\{\theta_0\}^c}(x) = 1$.
Conclude from monotonicity that,
for every $\theta \neq \theta_0$,
$\varphi_{\{\theta\}}(x) = 1$.
\end{proof}

\begin{proof}[Proof of Theorem \ref{thm:fullyconsistent_EC}]
 Since $\varphi$ is logically coherent, 
 it follows from \citet{Esteves2016} that
 there exists $R(x)$ such that,
 $\varphi_H(x) = 1 \Leftrightarrow H \cap R(x) = \emptyset$,
 $\varphi_H(x) = 0 \Leftrightarrow  R(x) \subseteq H$ and
 $\varphi_H(x) = \half$, otherwise.
 Using \cref{lemma:exisstenceagnostic}, conclude that
 $R(x) \neq \emptyset$.
 In the following, 
 we determine a loss, $L$, and 
 a joint probability, $\P(\theta,x)$,
 such that $\varphi$ is Bayes.
 
 Let $|\Theta| = k$.
 Also, let $L$ be the TEC given by \cref{tab:fullyconsistent_EC}.
 It follows from \citet{Yao2007} that
 $\varphi$ is Bayes with respect to $L$ when:
 \begin{align}
  \label{eq:thm:fullyconsistent_EC_1}
  \varphi_H(x)
  &= \begin{cases}
   1 & \text{, if } \P(\theta \in H|x) < \frac{1}{k} \\
   0 & \text{, if } \P(\theta \in H|x) > \frac{k-1}{k} \\
   \half & \text{, otherwise.}
  \end{cases}
 \end{align}
 Next, we determine $\P(\theta,x)$ such that
 these conditions hold.
 \begin{table}
  \caption{Loss function used in the proof of Theorem \ref{thm:fullyconsistent_EC}.}
  \centering
  \begin{tabular}{|l|c  p{1cm}  |}\hline
   \textbf{Decision} & \textbf{state of the nature} & \\ 
   & $\theta \in A$ & $\theta \notin A$ \\ \hline
   $0$ (accept A) & $0$ & $k$  \\
   ${\half}$ (remain agnostic about $A$)
   &  $1$  & $1$  \\ 
   $1$ (reject $A$)  & $k$  & $0$  \\ \hline
  \end{tabular}
  \label{tab:fullyconsistent_EC}
 \end{table}
 
 In order to determine $\P(\theta,x)$
 it is sufficient to choose $\P(x)$ and $\P(\theta|x)$.
 For each $A \subset \sX$, let 
 $\P(x \in A) = \frac{|A|}{|\sX|}$,
 that is, the uniform distribution over $\sX$.
 Also, for $H \subset \Theta$,
 \begin{align}
  \label{eq:thm:fullyconsistent_EC_2}
  \P(\theta \in H|x)
  &= \frac{1}{2k} 
  \cdot \frac{|H|}{|\sX|}
  + \frac{2k-1}{2k} 
  \cdot \frac{|H \cap R(x)|}{|R(x)|} .
 \end{align}
 It remains to show that
 $\varphi$ is Bayes with respect to
 $L$ and $\P$. We study three cases:
 \textbf{(i)} If $\varphi_H(x) = 1$, then 
 $H \cap R(x) = \emptyset$.
 Using \cref{eq:thm:fullyconsistent_EC_2},
 conclude that 
 $\P(\theta \in H|x) \leq 
 \frac{1}{2k} \cdot 1 + \frac{2k-1}{2k} \cdot 0 < \frac{1}{k}$,
 \textbf{(ii)} If $\varphi_H(x) = 0$, then
 $R(x) \subseteq H$.
 Using \cref{eq:thm:fullyconsistent_EC_2},
 conclude that
 $\P(\theta \in H|x) 
 \geq \frac{1}{2k} \cdot 0 + \frac{2k-1}{2k} \cdot 1 > \frac{k-1}{k}$,
 \textbf{(iii)} If $\varphi_H(x) = \half$, then
 $R(x) \cap H^c \neq \emptyset$ and
 $R(x) \cap H \neq \emptyset$, that is,
 $1 \leq |H \cap R(x)| < |R(x)| \leq k$.
 Using \cref{eq:thm:fullyconsistent_EC_2},
 conclude that
 $\P(\theta \in H|x) 
 \geq \frac{1}{2k} \cdot \frac{1}{k} + \frac{2k-1}{2k} \cdot \frac{1}{k} = \frac{1}{k}$. Also,
 $\P(\theta \in H|x) 
 \leq \frac{1}{2k} \cdot \frac{k-1}{k} + \frac{2k-1}{2k} \cdot \frac{k-1}{k} 
 = \frac{k-1}{k}$. That is,
 $\frac{1}{k} \leq \P(\theta \in H|x) \leq \frac{k-1}{k}$.
 It follows from \cref{eq:thm:fullyconsistent_EC_1} that
 $\varphi$ is Bayes with respect to $L$ using $\P$.
\end{proof}

\begin{lemma}
 \label{lemma:ec_logical}
 Let $L$ be an EC loss \cref{def:ec_loss} and,
 for each $\P$, let $\varphi_{\P,L}$ be
 a Bayes simultaneous test for $\P$ against $L$.
 If, for every $\P$, $\varphi_{\P,L}$ is
 logically coherent, then
 $\varphi_{\P,L}$ is a simultaneous test
 such as in \cref{ex:ec_loss} and:
 \begin{enumerate}
  \item for every $A, B \in \sigma(\Theta)$ 
  such that $\emptyset \neq A \subseteq B \neq \Omega$,
  $\alpha^A \geq \alpha^B$.
  \item for every $A, B \in \sigma(\Theta)$
  such that $A-B \neq \emptyset$, $B-A \neq \emptyset$,
  and $A \cup B \neq \Omega$:
  $\alpha^A + \alpha^B \leq \alpha^{A \cup B}$.
 \end{enumerate}
\end{lemma}

\begin{proof}
 Let $x \in \sX$ be arbitrary.
 
 If $\alpha^A < \alpha^B$, then for
 $\P$ such that 
 $\P(\theta \in A|x) = 
  \P(\theta \in B|x) =
  0.5(\alpha^A + \alpha^B)$,
  $\varphi_{\P,L}(A) < 1$ and
  $\varphi_{\P,L}(B) = 1$, that is,
 $\varphi_{\P,L}$ does not satisfy monotonicity.
 Conclude that, if $\varphi_{\P,L}$ is
 logically coherent for every $\P$, then
 $\alpha^A \geq \alpha^B$ for every 
 $\emptyset \neq A \subseteq B \neq \Omega$.
 
 If $\alpha^A + \alpha^B > \alpha^{A \cup B}$, then 
 let $\delta := (\alpha^A + \alpha^B) 
 - \alpha^{A \cup B} > 0$.
 By taking $\P$ such that
 \begin{align*}
  \P(\theta \in A|x) &= \max(0, \alpha^A - 0.4\delta), \\
  \P(\theta \in B|x) &= \max(0, \alpha^B - 0.4\delta), \\
  \P(\theta \in A \cup B|x) 
  &= \min(1, \alpha^A+\alpha^B-0.8\delta),
 \end{align*}
 obtain $\varphi_{\P,L}(A) = 1$, 
 $\varphi_{\P,L}(B) = 1$,
 and $\varphi_{\P,L}(A \cup B) < 1$, that is,
 it follows from \cref{lemma:union} that
 $\varphi_{\P,L}$ is not logically coherent.
 Conclude that, if $\varphi_{\P,L}$ is
 logically coherent for every $\P$, then
 $\alpha^A + \alpha^B \leq \alpha^{A \cup B}$.
\end{proof}

\begin{proof}[Proof of Theorem \ref{thm:impossibility}]
 Assume that, for every $\P$,
 $\varphi_{L, \P}$ is logically coherent.
 Let $\theta_1,\theta_2 \in \Theta$ and
 $A = \{\theta_1\}$, $B = \{\theta_2\}$.
 Since $|\Theta| \geq 3$,
 $A-B \neq \emptyset$, $B-A \neq \emptyset$ and
 $A \cup B \neq \Omega$. Hence, 
 it follows from \cref{lemma:ec_logical} that
 \begin{align*}
  \alpha^A &\geq \alpha^{A \cup B}, \\
  \alpha^B &\geq \alpha^{A \cup B}, \\
  \alpha^{A \cup B} &\geq \alpha^A + \alpha^B.
 \end{align*}
 That is, $\alpha^A = \alpha^B = \alpha^{A \cup B} = 0$,
 a contradiction with \cref{ex:ec_loss}.
 Conclude that there exists $\P$ such that
 $\varphi_{L,\P}$ is not logically coherent.
\end{proof}

\begin{proof}[Proof of Theorem \ref{thm:gfbstloss}]
The posterior expected losses for each decision
are given by:
\begin{align*}
  \E\left[L_A (0,(\theta,x) | x\right]
  &= b \P(\theta \notin T_x^A \cup T_x^{A^c} | x) + (b + c) \P(\theta \in T_x^A | x),  \\  
  \E\left[L_A \left(\half,(\theta,x\right) \bigg| x\right] 
  &= v + c\P (\theta \in T_x^A \cup T_x^{A^c} | x), \\
  \E\left[L_A (1,(\theta,x) | x\right]
  &= b \P (\theta \notin T_x^A \cup T_x^{A^c} | x)  \ + \   (b + c) \P (\theta \in T_x^{A^c} | x)  \ . 
\end{align*}
Next, it follows from definition that
$T_x^A \subseteq A^c$ and
$T_x^{A^c} \subseteq A$. Hence,
$T_x^A \cap T_x^{A^c} = \emptyset$.
Hence,
\begin{align*}
 \E\left[L_A (0,(\theta,x) | x\right]
 - \E\left[L_A \left(\half,(\theta,x\right) \bigg| x\right]
 &= (b+c)\P(\theta \notin T_{x}^{A^c}|x) - (v+c) \\
 \E\left[L_A (0,(\theta,x) | x\right]
 -  \E\left[L_A (1,(\theta,x) | x\right]
 &= (b+c)\left(\P(\theta \in T_x^{A}|x)
 -\P(\theta \in T_x^{A^c}|x)\right) \\
 \E\left[L_A (1,(\theta,x) | x\right]
 -\E\left[L_A \left(\half,(\theta,x\right) \bigg| x\right]
 &= (b+c)\P(\theta \notin T_{x}^{A}|x) - (v+c)
\end{align*}
Also, recall from definition that
either $T_x^{A} = \emptyset$ or
$T_x^{A^c} = \emptyset$.
Hence, since $0 < v < b$ and $c > 0$,
if $\varphi$ is Bayes, then
$\varphi_H(x) = 0$ if and only if
$\P(\theta \notin T_{x}^{A^c}|x) < \frac{v+c}{b+c}$ and
$\varphi_H(x) = 1$ if and only if
$\P(\theta \notin T_{x}^{A}|x) < \frac{v+c}{b+c}$.
It follows from \citet{Esteves2016} that
$\varphi$ is the GFBST.
\end{proof}